\documentclass[12pt]{amsart}

\usepackage[margin = 1in]{geometry}
\usepackage{amsfonts, amsmath, amscd, amssymb,
latexsym, mathrsfs, mathtools, 
slashed, stmaryrd, verbatim, wasysym, bbm, url}
\usepackage[utf8]{inputenc}
\usepackage[backend=bibtex,bibstyle=authoryear,citestyle=alphabetic,giveninits=true,doi=false,isbn=false,url=false,backref]{biblatex}
\usepackage[usenames]{xcolor}
\usepackage{enumerate}
\usepackage[shortlabels]{enumitem}

\DeclareFieldFormat{labelalphawidth}{\mkbibbrackets{#1}}

\defbibenvironment{bibliography}
  {\list
     {\printtext[labelalphawidth]{%
        \printfield{labelprefix}%
    \printfield{labelalpha}%
        \printfield{extraalpha}}}
     {\setlength{\labelwidth}{\labelalphawidth}%
      \setlength{\leftmargin}{\labelwidth}%
      \setlength{\labelsep}{\biblabelsep}%
      \addtolength{\leftmargin}{\labelsep}%
      \setlength{\itemsep}{\bibitemsep}%
      \setlength{\parsep}{\bibparsep}}%
      }
  {\endlist}
  {\item}

\newtheorem*{acknowledgements}{Acknowledgements}
\newtheorem*{theorem*}{Theorem}
\newtheorem{theorem}{Theorem}
\newtheorem{corollary}{Corollary}

\newtheorem{lemma}[corollary]{Lemma}
\newtheorem{proposition}[corollary]{Proposition}
\theoremstyle{definition}
\newtheorem{definition}{Definition}
\newtheorem{example}{Example}
\newtheorem{remark}[example]{Remark}

\numberwithin{equation}{section}

\let\oldsqrt\sqrt
\def\sqrt{\mathpalette\DHLhksqrt}
\def\DHLhksqrt#1#2{%
\setbox0=\hbox{$#1\oldsqrt{#2\,}$}\dimen0=\ht0
\advance\dimen0-0.2\ht0
\setbox2=\hbox{\vrule height\ht0 depth -\dimen0}%
{\box0\lower0.4pt\box2}}

\usepackage{verbatim} 
\DeclareFontFamily{U}{mathx}{\hyphenchar\font45}
\DeclareFontShape{U}{mathx}{m}{n}{
      <5> <6> <7> <8> <9> <10>
      <10.95> <12> <14.4> <17.28> <20.74> <24.88>
      mathx10
      }{}
\DeclareSymbolFont{mathx}{U}{mathx}{m}{n}
\DeclareFontSubstitution{U}{mathx}{m}{n}
\DeclareMathAccent{\widecheck}{0}{mathx}{"71}

\renewcommand{\tilde}[1]{\widetilde{#1}}

\renewcommand{\hat}[1]{\widehat{#1}}

\newcommand\eps\varepsilon
\renewcommand\epsilon\varepsilon

\newcommand{\abs}[1]{\left\lvert #1 \right\rvert}
\newcommand{\smallabs}[1]{\lvert #1 \rvert}

\newcommand{\norm}[1]{\lVert #1 \rVert}
\newcommand\inner[1]{\langle #1 \rangle}

\newcommand\Mand{\text{ and }}

\newcommand\paperintro%
        {%
         }
\newcommand\paperbody%
        {%
         }

\newcommand\bbG{\mathbb{G}}
\newcommand\bbH{\mathbb{H}}

\newcommand\bbR{\mathbb{R}}

\newcommand\bbZ{\mathbb{Z}}

\newcommand\cA{\mathcal{A}}

\newcommand\cD{\mathcal{D}}
\newcommand\cE{\mathcal{E}}

\newcommand\cI{\mathcal{I}}

\newcommand\cU{\mathcal{U}}

\newcommand\mf[1]{\mathfrak{ #1}}

\DeclareMathAlphabet{\mathpzc}{OT1}{pzc}{m}{it}

\newcommand{\sbs}{\subset}

\usepackage[refpage]{nomencl}

\makenomenclature

\usepackage{tikz, tikz-cd, tkz-euclide}
\usetikzlibrary{calc,positioning,intersections,quotes,decorations.markings}
\makeatletter
\def\@tocline#1#2#3#4#5#6#7{\relax
  \ifnum #1>\c@tocdepth 
  \else
    \par \addpenalty\@secpenalty\addvspace{#2}%
    \begingroup \hyphenpenalty\@M
    \@ifempty{#4}{%
      \@tempdima\csname r@tocindent\number#1\endcsname\relax
    }{%
      \@tempdima#4\relax
    }%
    \parindent\z@ \leftskip#3\relax \advance\leftskip\@tempdima\relax
    \rightskip\@pnumwidth plus4em \parfillskip-\@pnumwidth
    #5\leavevmode\hskip-\@tempdima
      \ifcase #1
       \or\or \hskip 1em \or \hskip 2em \else \hskip 3em \fi%
      #6\nobreak\relax
    \hfill\hbox to\@pnumwidth{\@tocpagenum{#7}}\par
    \nobreak
    \endgroup
  \fi}
\makeatother

\def\annu#1{_{%
  \vbox{\hrule height .2pt 
    \kern 1pt 
    \hbox{$\scriptstyle {#1}\kern 1pt$}%
  }\kern-.05pt 
  \vrule width .2pt 
}}

\allowdisplaybreaks
\usepackage[colorlinks,citecolor=blue,urlcolor=red,bookmarks=false,hypertexnames=true]{hyperref} 

\makeatletter
\def\keywords{\xdef\@thefnmark{}\@footnotetext}
\makeatother

\title[Optimal subelliptic super-Poincar\'e and isoperimetric inequalities]{Optimal subelliptic super-Poincar\'e and isoperimetric inequalities on stratified Lie groups}
\author{Yaozhong Qiu}
\address{Department of Mathematics, Imperial College London, 180 Queen’s Gate, London, SW7 2AZ, United Kingdom}
\email{y.qiu20@imperial.ac.uk}

\begin{document}
\begin{abstract}
We prove $q$-super-Poincar\'e inequalities, $q \in [1, 2]$, for a class of exponential power type probability measures defined in terms of a norm in a number of subelliptic settings, primarily on stratified Lie groups but also in the Grushin and Heisenberg-Greiner settings. Our results include generically optimal isoperimetric inequalities for such probability measures. 
\end{abstract}


\maketitle

\section{Introduction and main results}\label{S1}

In this paper, we prove super-Poincar\'e and isoperimetric inequalities for a class of exponential power type probability measures defined in terms of a norm in a number of subelliptic settings. The most striking feature of our results is that while our motivating examples of measures have supergaussian tails, the isoperimetric inequality they satisfy, which is generically optimal up to constants, corresponds to that of the subgaussian exponential power type measure $d\nu_r = Z^{-1}e^{-\abs{x}^r}dx$, $r \in (1, 2)$ regarded as a measure on the metric space $(\bbR^n, \abs{\cdot})$.  

If $(X, \cA, \mu)$ is a probability space and $L$ is a selfadjoint operator on $L^2(\mu)$ generating a Markov semigroup $(P_t)_{t \geq 0}$, then the $2$-super-Poincar\'e inequality introduced in \cite{wang2000functional, wang2002functional} is the family of inequalities 
\begin{equation}\label{2spi0}
    \int_X f^2d\mu \leq \epsilon \cE(f, f) + \beta_2(\epsilon)\left(\int_X \abs{f}d\mu\right)^2, \quad \epsilon > \epsilon_0 \geq 0, 
\end{equation}
where $f: X \rightarrow \bbR$ belongs to the domain $\cD(\cE)$ of the Dirichlet form $\cE$ associated to $L$ and $\beta_2: (\epsilon_0, \infty) \rightarrow \bbR$ is a nondecreasing function which we call the growth. As is to be expected in the theory of functional inequalities, \eqref{2spi0} has spectral theoretic content for $L$, ergodicity implications for $(P_t)_{t \geq 0}$, and enjoys relationships with other functional inequalities. In particular, \cite[Theorem~2.1]{wang2000functional} asserts \eqref{2spi0} is equivalent to the containment $\sigma_{\text{ess}}(-L) \sbs [\epsilon_0^{-1}, \infty)$ of the essential spectrum of $-L$, so that if $\epsilon_0 = 0$ then $-L$ has discrete spectrum. Moreover, according to \cite[Equation~1.6]{wang2000functional}, an analogue of the equivalence between the Poincar\'e inequality and exponential decay to equilibrium (see \cite[Theorem~4.2.5]{bakry2014analysis}) holds, namely \eqref{2spi0} holds if and only if
\[ \int_X (P_tf)^2d\mu \leq e^{-2t/r}\int_X f^2d\mu + \beta_2(r)(1 - e^{-2t/r})\left(\int_X \abs{f}d\mu\right)^2. \]
Lastly, \cite[Theorems~3.1~and~3.2]{wang2000functional} assert \eqref{2spi0} is equivalent to a family of defective $2$-$F$-Sobolev inequalities 
\begin{equation}\label{2fsob0}
    \int_X f^2F(f^2)d\mu \leq c_1\cE(f, f) + c_2, \quad \int_X f^2d\mu = 1 \Mand c_1, c_2 > 0.
\end{equation}
When $\beta_2(\epsilon) = e^{C\epsilon^{-\sigma}}$ for some $C > 0$, the case $\sigma = 1$ implies through \eqref{2fsob0} the defective logarithmic Sobolev inequality, see \cite[\S5]{bakry2014analysis}, while the case $\sigma > 1$ implies the Lata{\l}a-Oleszkiewicz inequality of \cite{latala2000between}, see also \cite[\S7.6.3]{bakry2014analysis}. We refer the reader to \cite{bakry2004functional, wang2006functional, bakry2014analysis} for more details on functional inequalities and their applications to spectral theory and geometry. 

It is known that under certain curvature lower bounds that a priori weaker functional inequalities such as the ($2$-super-)Poincar\'e inequality and logarithmic Sobolev inequalities can contain isoperimetric content. For instance, with respect to the classical Bakry-\'Emery curvature of \cite{bakry1985diffusions} and the generalisation developed by \cite{baudoin2012log, baudoin2014sub, baudoin2016curvature} adapted to the subelliptic setting, this is the content of \cite[Theorem~3.4]{wang2000functional} and \cite[Theorem~1.6]{baudoin2012log} respectively. However, the methods we shall use will sidestep curvature considerations entirely and instead we will prove an $L^1$-analogue of \eqref{2spi0}. It is traditional that such $L^1$-functional inequalities have isoperimetric content which shall be contained in Theorem \ref{thm1}. 

While a natural setting for our discussion is that of a stratified Lie group, in that the motivating example is a probability measure defined on the Heisenberg group, the study of the Grushin and Heisenberg-Greiner settings technically fall outside their scope and so we present our results first in general terms even if the examples we consider in the sequel have more structure than that which is necessary. Our setting will follow \cite{inglis2011u}. We consider the metric measure space $(\bbR^n, d, \mu)$ equipped with a subgradient $\nabla = (X_1, \cdots, X_\ell)$ and sublaplacian $\Delta = \nabla \cdot \nabla = \sum_{i=1}^\ell X_i^2$ where the $X_1, \cdots, X_\ell$ are a finite collection of possibly noncommuting vector fields whose divergence with respect to the Lebesgue measure $\xi$ on $\bbR^n$ vanishes. We assume $d$ is related to $\nabla$ through the formula
\begin{equation}\label{d-nabla}
    \abs{\nabla f(x)} = \limsup_{d(x, y) \rightarrow 0^+} \frac{\abs{f(x) - f(y)}}{d(x, y)}.
\end{equation}

\begin{definition}
    For a Borel set $A \sbs \bbR^n$, the surface measure $\mu^+(A)$ of $A$ is defined by 
    \[ \mu^+(A) = \liminf_{\epsilon \rightarrow 0^+} \frac{\mu(A_\epsilon) - \mu(A)}{\epsilon}, \quad A_\epsilon = \{y \in \bbR^n \mid d(y, A) < \epsilon\} \] 
    and the isoperimetric profile of $\mu$ by 
    \[ I_\mu(t) = \inf \{\mu^+(A) \mid A \text{ Borel and such that } \mu(A) = t\}. \]    
\end{definition}

Our measures are exponential power type of the form 
\begin{equation}\label{measure}
    d\mu = Z^{-1}e^{-N^p}d\xi
\end{equation}
where $N: \bbR^n \rightarrow \bbR$ is a sufficiently smooth function satisfying a family of estimates given in the sequel and $Z = Z_p$ normalises $\mu = \mu_p$. In the sequel, we will always interpret $Z$ as the normalisation constant for a given measure wherever it appears. 

In the sequel, all integration happens over $\bbR^n$ unless otherwise specified. We will be interested in proving the following $L^1$-analogue 
\begin{equation}\label{1spi}
    \int \abs{f}d\mu \leq \epsilon \int \abs{\nabla f}d\mu + \beta_1(\epsilon)\left(\int \abs{f}^{1/2} d\mu\right)^2, \quad \epsilon > 0  
\end{equation}
of \eqref{2spi0}, which we call the $1$-super-Poincar\'e inequality, and where $f: \bbR^n \rightarrow \bbR$ again belongs to the domain $\cD(\cE) = W^{1, 2}(\mu)$ of the Dirichlet form $\cE(f, g) = \int \nabla f \cdot \nabla g d\mu$. Our proof of the isoperimetric inequality for $\mu$ will go through the $L^1$-analogue 
\begin{equation}\label{1fsob}
    \int \abs{f}F(\abs{f})d\mu \leq c_1\int \abs{\nabla f}d\mu + c_2, \quad \int \abs{f}d\mu = 1 \Mand c_1, c_2 > 0 
\end{equation}
of \eqref{2fsob0} which we call similarly the $1$-$F$-Sobolev inequality. Lastly, we write $f \lesssim g$ for two functions $f$ and $g$ if $f \leq Cg$ for some constant $C > 0$, similarly for $f \gtrsim g$, and $f \simeq g$ for $f \lesssim g \lesssim f$. We also write $\Phi \lesssim \Psi$ for two functionals $\Phi$ and $\Psi$ if $\Phi(f) \leq C\Psi(f)$ for some constant $C > 0$ independent of $f$, similarly for $\Phi \gtrsim \Psi$, and $\Phi \simeq \Psi$ if $\Phi \lesssim \Psi \lesssim \Phi$. 

\begin{theorem}\label{thm1}
    Assume $N$ satisfies in the distributional sense 
    \begin{equation}\label{estimates}
        \frac{\abs{x}^\alpha}{N^\alpha} \lesssim \abs{\nabla N} \lesssim \frac{\abs{x}^\alpha}{N^\alpha}, \quad \Delta N \lesssim \frac{\abs{x}^{2\alpha}}{N^{2\alpha+1}}, \quad \nabla N \cdot \nabla \abs{x} \lesssim \frac{\abs{x}^{2\alpha+1}}{N^{2\alpha+1}} 
    \end{equation}
    for some $\alpha > 0$ and where $\abs{x}$ is the euclidean norm of a generic point $\xi \in \bbR^n$ written in the form $\xi = (x, x') \in \bbR^{n_1} \times \bbR^{n_2}$ with $n_1 \geq 2$. Assume moreover $\abs{x}$ satisfies again in the distributional sense 
    \begin{equation}\label{laplacian}
        \nabla \abs{x} \lesssim 1, \quad \Delta \abs{x} = \frac{n_1 - 1}{\abs{x}}, \quad \nabla \cdot x = n_1,
    \end{equation}
    and 
    \begin{equation}\label{comparison}
        \abs{x} \lesssim N, \quad d \lesssim N \lesssim d.
    \end{equation}
    If there exists a local $1$-super-Poincar\'e inequality 
    \begin{equation}\label{1spilocal}
        \int \abs{f}d\xi \leq \delta \int \abs{\nabla f}d\xi + \tilde{\beta}_1(\delta)\left(\int \abs{f}^{1/2}d\xi\right)^2, \quad \tilde{\beta}_1(\delta) \lesssim 1 + \delta^{-Q}
    \end{equation}
    for some $Q > 0$ and $f \in W^{1, 2}(\mu)$, and a local $1$-Poincar\'e inequality 
    \begin{equation}\label{1poincarelocal}
        \int_{B_R} \abs{f - \frac{1}{\xi(B_R)} \int_{B_R} fd\xi} d\xi \lesssim \int_{B_R} \abs{\nabla f}d\xi
    \end{equation}
    for $f \in W^{1,2}(B_R)$ where $B_R$ is the $d$-ball of radius $R > 0$ centred at the origin, then for each $p \geq \alpha + 1$, the measure $\mu$ defined according to \eqref{measure} satisfies the isoperimetric inequality
    \[ I_\mu \gtrsim \cU_r, \quad r = \frac{(\alpha + 1)p}{(\alpha + 1) + \alpha p} \]
    with respect to the metric $d$ and where $\cU_r(t)$ is the isoperimetric function of the measure $d\nu_r = Z^{-1}e^{-\abs{x}^r}dx$ with respect to the euclidean metric $\abs{\cdot}$ on $\bbR^n$. 
\end{theorem}

Our motivating examples live on a stratified Lie group (discussed in the sequel) with subgradient $\nabla$ for which the $L^1$-Sobolev inequality of \cite[Theorem~IV.7.1]{varopoulos1991analysis} readily implies \eqref{1spilocal} by H\"older's inequality and linearisation, while the local $1$-Poincar\'e inequality \eqref{1poincarelocal} follows from a generalisation of \cite[Theorem~2.1]{jerison1986poincare}. Our functions $N$ will satisfy the conditions of Theorem \ref{thm1} with $\alpha = 1$. This class of measures therefore enjoy supergaussian decay of tails but satisfy the isoperimetric inequality of the subgaussian model measure $d\nu_{2p/(p+2)}$. The isoperimetric inequality of Theorem \ref{thm1} is a priori not necessarily optimal (in the sense $I_\mu \gtrsim \cU_r$ does not preclude the possibility that $I_\mu \gtrsim \cU_{r+\epsilon}$ for some $\epsilon > 0$) but indeed the isoperimetric inequality we prove is optimal for some measures.
\section{Proof of main results}

Our proof will follow the ideas of \cite{hebisch2010coercive} of using $U$-bounds, which are inequalities of the form 
\begin{equation}\label{ubound0}
    \int \abs{f}^qUd\mu \lesssim \int \abs{\nabla f}^qd\mu + \int \abs{f}^qd\mu
\end{equation}
for some function $U: \bbR^n \rightarrow \bbR$ satisfying some growth conditions and $q \in [1, 2]$. Only the case $q = 1$ is needed to achieve the isoperimetric inequality of Theorem \ref{thm1}, but we prove nevertheless the $q$-super-Poincar\'e inequality for each $q \in [1, 2]$ to complete the picture. The method of $U$-bounds were initiated in \cite{hebisch2010coercive} and used to prove a number of functional inequalities, including the $q$-Poincar\'e and $q$-logarithmic Sobolev inequalities (improving on the case $q = 2$) introduced in \cite{bobkov2005entropy} and weighted versions thereof, and also the logarithmic${}^\theta$-Sobolev inequalities (where the entropy functional $f\log f$ on the left hand side is replaced by $f\log(1 + f)^\theta$) which are related to super-Poincar\'e inequalities through the equivalence described by \cite[Theorems~3.1~and~3.2]{wang2000functional}. For $q = 2$, such inequalities can be obtained in a number of ways; for instance, one is via an expansion of the square argument as observed in \cite[\S2.3]{roberto2022hypercontractivity}, a second is based on the Bakry-\'Emery calculus for multiplication operators studied in \cite{roberto2021bakry, qiu2024bakry}, and another is through the connection between diffusion and Schr\"odinger operators, see \cite[\S1.15.7]{bakry2014analysis}. More generally they can be obtained for each $q \in [1, 2]$ via an integration by parts argument. 

The method of $U$-bounds, as far as applications to other functional inequalities are concerned, are somewhat qualitatively similar to the method of Lyapunov functions pushed forward by the series of works \cite{bakry2008rate, cattiaux2009lyapunov, cattiaux2010functional, cattiaux2010functional2, cattiaux2010note, cattiaux2011some}. For instance, both methods rely essentially on the existence of a localised version of the desired inequality and both $U$-bounds and Lyapunov functions enter the proofs in a neighbourhood of infinity. Moreover, both yield explicit constants and are generally quite robust with respect to (suitable) unbounded perturbations of the measure. In any case, one reason for working with $U$-bounds instead of Lyapunov functions is that there is more literature for $U$-bounds in the subelliptic settings we consider, see, for instance, the series of works \cite{inglis2012spectral, bou2021coercive, dagher2022coercive, chatzakou2022poincare, chatzakou2023q, bou2024coercive} and also \cite{inglis2009logarithmic, papageorgiou2014logarithmic, papageorgiou2018log, inglis2019log, dagher2022spectral, konstantopoulos2025log} for applications in the infinite-dimensional setting.

We start by proving a $U$-bound for our measure $\mu$ with a function $U$ which does not satisfy the classical growth conditions of \cite{hebisch2010coercive} but instead degenerates along $\{\abs{x} = 0\}$ as in \cite{inglis2012spectral}. Then we prove a Hardy inequality, which can also be regarded as a $U$-bound, for $\mu$ and show that the degeneracy, for $p > \alpha + 1$ can be resolved as a result, which allows passage to \eqref{1spi} and equivalently \eqref{1fsob}. Finally, we conclude by arguments of \cite{inglis2011u} the desired isoperimetric inequality. 

\subsection{The case $p > \alpha + 1$}

\begin{lemma}\label{lem:ubound}
    Under the assumptions of Theorem \ref{thm1}, the measure $\mu$ satisfies the $U$-bound 
    \begin{equation}\label{ubound1}
        \int \abs{f}^qU_qd\mu \vcentcolon= \int \abs{f}^q\abs{x}^{q\alpha}N^{q(p-\alpha-1)}d\mu \lesssim \int \abs{\nabla f}^qd\mu + \int \abs{f}^qd\mu
    \end{equation}
    for each $q \in [1, 2]$. 
\end{lemma}

\begin{proof}
    Following the arguments of \cite{hebisch2010coercive}, we develop by Leibniz rule firstly
    \begin{equation}\label{uboundproof1}
    \begin{aligned}
        \int \nabla(\abs{f}^qe^{-N^p}) \cdot \abs{x}^s N^t \nabla N d\xi &= q \int \abs{f}^{q-2} f \abs{x}^s N^t \nabla f \cdot \nabla N d\mu \\
        &- p \int \abs{f}^q \abs{x}^s N^{t+p-1} \abs{\nabla N}^2 d\mu
    \end{aligned}
    \end{equation}
    for some $s, t > 0$ to be determined later, and again by integrating by parts secondly 
    \begin{equation}\label{uboundproof2}
    \begin{aligned}
        \int \nabla(\abs{f}^qe^{-N^p}) \cdot \abs{x}^s N^t \nabla N d\xi &= -\int \abs{f}^q \nabla \cdot \abs{x}^s N^t \nabla N d\mu \\
        &= -s \int \abs{f}^q \abs{x}^{s-1}N^t \nabla N \cdot \nabla \abs{x} d\mu \\
        &\phantom{=} -t \int \abs{f}^q \abs{x}^s N^{t-1} \abs{\nabla N}^2 d\mu \\ 
        &\phantom{=} - \int \abs{f}^q\abs{x}^s N^t \Delta N d\mu. 
    \end{aligned}
    \end{equation}
    With $s = \alpha(q - 2)$ and $t = (q - 1)(p - \alpha - 1) + \alpha$ we obtain the upper bound on \eqref{uboundproof1}
    \begin{align*}
        \int \nabla(\abs{f}^qe^{-N^p}) \cdot \abs{x}^s N^t \nabla N d\xi &\lesssim \int \abs{f}^{q-2} f \abs{\nabla f} \abs{x}^{s+\alpha} N^{t-\alpha} d\mu - \int \abs{f}^q \abs{x}^{s+\alpha} N^{t+p-1+\alpha} d\mu \\
        &\lesssim \epsilon \int \abs{f}^qU_qd\mu + \frac{1}{\epsilon} \int \abs{\nabla f}^qd\mu - \int \abs{f}^qU_qd\mu 
    \end{align*}
    by using the upper and lower estimates on $\abs{\nabla N}^2$ given by \eqref{estimates} and Young's inequality provided $q > 1$, and otherwise if $q = 1$ then Young's inequality can be sidestepped since $q - 1 = s + \alpha = t - \alpha = 0$. On the other hand, the weight in each of the three integrals on the right hand side of \eqref{uboundproof2} is controlled by $\abs{x}^{s+2\alpha}N^{t-2\alpha-1} = U_q/N^p$ again by \eqref{estimates}. For $\epsilon > 0$ sufficiently small this implies
    \begin{equation}\label{criticality1}
    \begin{aligned}
        \int \abs{f}^qU_qd\mu &\lesssim \int \abs{\nabla f}^qd\mu + \int \abs{f}^qU_q/N^pd\mu \\
        &\lesssim \int \abs{\nabla f}^qd\mu + \epsilon' \int \abs{f}^qU_qd\mu + K(\epsilon') \int \abs{f}^qd\mu 
    \end{aligned}
    \end{equation}
    which gives \eqref{ubound1} for another $\epsilon' > 0$ sufficiently small. Indeed, by continuity of $U_q/N^p$, we may bound $\int_{B_R} \abs{f}^qU_q/N^pd\mu \lesssim \int \abs{f}^qd\mu$ on a $d$-ball $B_R$ of radius $R > 0$ and then bound $\int_{B_R^c} \abs{f}^qU_q/N^pd\mu \lesssim \epsilon' \int \abs{f}^qU_qd\mu$ on $B_R^c$ by taking $R$ sufficiently large. 
\end{proof}

As claimed earlier, $U_q$ vanishes along $\{\abs{x} = 0\}$ which has codimension $n_1 \geq 2$. Since $\Delta \abs{x} = \frac{n_1 - 1}{\abs{x}}$, we now show a Hardy inequality with singularity along $\{\abs{x} = 0\}$ holds. 

\begin{lemma}\label{lem:hardy}
    Under the assumptions of Theorem \ref{thm1}, the measure $\mu$ satisfies the Hardy inequality 
    \begin{equation}\label{hardy1}      
        \int \frac{\abs{f}^q}{\abs{x}^q}d\mu \lesssim \int \abs{\nabla f}^qd\mu + \int \abs{f}^qd\mu
    \end{equation}
    for each $q \in [1, 2)$, and the Caffarelli-Kohn-Nirenberg inequality
    \begin{equation}\label{uncertainty1}
        \int \abs{f}^qd\mu \lesssim \left(\int \abs{f}^q\abs{x}^{q\alpha}d\mu\right)^{1/(\alpha+1)} \left(\int \abs{\nabla f}^qd\mu + \int \abs{f}^qd\mu\right)^{\alpha/(\alpha+1)} 
    \end{equation} 
    for each $q \in [1, 2]$. 
\end{lemma}

\begin{remark}
    The case $q = 2$ in \eqref{hardy1} is excluded only since it may be the case $n_1 = 2$. If $n_1 > 2$ strictly, then \eqref{hardy1} holds. In any case, we show passage to the $2$-super-Poincar\'e inequality is not obstructed by this technical issue and that \eqref{uncertainty1} suffices. 
\end{remark}

\begin{proof}
    Writing the measure $\mu$ as $d\mu = Z^{-1}e^{-V}d\xi$ for $V = N^p$, we integrate by parts
    \begin{equation}\label{ibp0}
        \int \abs{f}^q\omega \nabla \cdot h d\mu = - q\int \abs{f}^{q-2}f \omega \nabla f \cdot hd\mu + \int \abs{f}^q\omega \nabla V \cdot hd\mu - \int \abs{f}^q\nabla \omega \cdot h d\mu 
    \end{equation}
    where $h$ and $\omega$ are a sufficiently regular vector field and weight respectively. To achieve \eqref{hardy1}, take $\omega \equiv 1$ and $h = \nabla \abs{x}^{2-q}$ so that $\nabla \cdot h = (2 - q)(n_1 - q)\abs{x}^{-q}$ by \eqref{laplacian} and therefore 
    \begin{align*}
        \int \frac{\abs{f}^q}{\abs{x}^q}d\mu &\lesssim \int \abs{f}^{q-2}f \nabla f \cdot \nabla \abs{x}^{2-q} d\mu + \int \abs{f}^q \nabla N^p \cdot \nabla \abs{x}^{2-q} d\mu \\
        &\lesssim \epsilon \int \frac{\abs{f}^q}{\abs{x}^q}d\mu + \frac{1}{\epsilon} \int \abs{\nabla f}^qd\mu + \int \abs{f}^q \abs{x}^{2(\alpha+1)-q}N^{p-2(\alpha+1)} d\mu \\
        &\lesssim \epsilon \int \frac{\abs{f}^q}{\abs{x}^q}d\mu + \frac{1}{\epsilon} \int \abs{\nabla f}^qd\mu + \int \abs{f}^q \abs{x}^{q\alpha}N^{p-q(\alpha+1)} d\mu
    \end{align*}
    by \eqref{estimates} and \eqref{comparison} and Young's inequality if $q > 1$. Since $p - q(\alpha + 1) \leq q(p - \alpha - 1)$, following previous arguments we can control the final addend by $\int \abs{f}^qU_qd\mu + \int \abs{f}^qd\mu$ which by Lemma \ref{lem:ubound} yields \eqref{hardy1} by taking $\epsilon$ sufficiently small. If $q = 1$ then
    \begin{equation}\label{criticality2}
    \begin{aligned}
        \int \frac{\abs{f}}{\abs{x}}d\mu &\lesssim \int \abs{f}^{-1}f\nabla f \cdot \nabla \abs{x}d\mu + \int \abs{f}\nabla N^p \cdot \nabla \abs{x}d\mu \\
        &\lesssim \int \abs{\nabla f}d\mu + \int \abs{f}\abs{x}^{2\alpha+1}N^{p-2(\alpha+1)}d\mu \\
        &\lesssim \int \abs{\nabla f}d\mu + \int \abs{f}U_1d\mu
    \end{aligned}
    \end{equation}
    which gives \eqref{hardy1}. Note to achieve \eqref{hardy1} in the critical case $q = 2$, provided $n_1 > 2$, we take $h = \nabla \log \abs{x}$. 

    To achieve \eqref{uncertainty1} in the case $q > 1$, take $\omega \equiv 1$ and $h = x$ so that $\nabla \cdot h = x$ by \eqref{laplacian} and therefore 
    \begin{equation}\label{uncertainty2}
        \begin{aligned}
            n_1 \int \abs{f}^qd\mu &= -q \int \abs{f}^{q-2}f \nabla f \cdot xd\mu + \int \abs{f}^q \nabla V \cdot xd\mu \vphantom{\left(\int \abs{f}^q\abs{x}^{\frac{q}{q-1}}d\mu\right)^{\frac{q-1}{q}} \left(\left(\int \abs{\nabla f}^qd\mu\right)^{1/q} + \left(\int \abs{f}^q\abs{\nabla V}^qd\mu\right)^{\frac{1}{q}}\right)} \\
            &\lesssim \left(\int \abs{f}^q\abs{x}^{\frac{q}{q-1}}d\mu\right)^{\frac{q-1}{q}} \left(\left(\int \abs{\nabla f}^qd\mu\right)^{1/q} + \left(\int \abs{f}^q\abs{\nabla V}^qd\mu\right)^{\frac{1}{q}}\right) \\
            &\lesssim \left(\int \abs{f}^q\abs{x}^{\frac{q}{q-1}}d\mu\right)^{\frac{q-1}{q}} \left(\int \abs{\nabla f}^qd\mu + \int \abs{f}^qd\mu\right)^{\frac{1}{q}} \vphantom{\left(\int \abs{f}^q\abs{x}^{\frac{q}{q-1}}d\mu\right)^{\frac{q-1}{q}} \left(\left(\int \abs{\nabla f}^qd\mu\right)^{1/q} + \left(\int \abs{f}^q\abs{\nabla V}^qd\mu\right)^{\frac{1}{q}}\right)}
        \end{aligned}
    \end{equation}
    by Lemma \ref{lem:ubound} and since $\abs{\nabla V}^q \simeq U_q$ by \eqref{estimates}. Similarly, take also $\omega = \abs{x}^{\alpha(q-1)-1}$ and again $h = x$ to obtain 
    \begin{equation}\label{uncertainty3}
        C(n_1, \alpha, q) \int \frac{\abs{f}^q}{\abs{x}^{1-\alpha(q-1)}}d\mu \lesssim \left(\int \abs{f}^q\abs{x}^{q\alpha}d\mu\right)^{\frac{q-1}{q}} \left(\int \abs{\nabla f}^qd\mu + \int \abs{f}^qd\mu\right)^{\frac{1}{q}}
    \end{equation}
    where $C(n_1, \alpha, q) = n_1 - 1 + \alpha(q - 1) > 0$. If $\alpha = \frac{1}{q-1}$ then \eqref{uncertainty2} and \eqref{uncertainty3} are equal and exactly \eqref{uncertainty1}. If $\alpha > \frac{1}{q-1}$ then by H\"older's inequality with exponent $\alpha(q-1) > 1$ we find 
    \[ \int \abs{f}^q\abs{x}^{\frac{q}{q-1}}d\mu \lesssim \left(\int \abs{f}^q\abs{x}^{q\alpha}d\mu\right)^{\frac{1}{\alpha(q-1)}}\left(\int \abs{f}^qd\mu\right)^{1-\frac{1}{\alpha(q-1)}} \]
    which implies \eqref{uncertainty1} through \eqref{uncertainty2}. If $\alpha < \frac{1}{q-1}$ then H\"older's inequality again with exponent $\frac{\alpha+1}{\alpha q} > 1$ we find 
    \[ \int \abs{f}^qd\mu = \int \abs{f}^q \frac{\abs{x}^{q\alpha(1 - \alpha(q-1))/(\alpha+1)}}{\abs{x}^{q\alpha(1 - \alpha(q-1))/(\alpha+1)}}d\mu \lesssim \left(\int \frac{\abs{f}^q}{\abs{x}^{1-\alpha(q-1)}}d\mu\right)^{\frac{q\alpha}{\alpha + 1}}\left(\int \abs{f}^q\abs{x}^{q\alpha}d\mu\right)^{1- \frac{q\alpha}{\alpha + 1}} \]
    which implies \eqref{uncertainty1} through \eqref{uncertainty3}. Note the earlier assumption $q > 1$ enters through H\"older's inequality and that it holds even for $n_1 = 1$ if $\alpha > 0$ and otherwise is immediate if $\alpha = 0$. The case $q = 1$ follows from \eqref{hardy1} and H\"older's inequality. 
\end{proof}

We now show how to use this Hardy inequality and Caffarelli-Kohn-Nirenberg inequality together with the $U$-bound to attain the super-Poincar\'e inequality.

\begin{proposition}\label{prop:spi}
    Under the assumptions of Theorem \ref{thm1}, the measure $\mu$ satisfies the $q$-super-Poincar\'e inequality 
    \begin{equation}\label{qspi}
        \int \abs{f}^qd\mu \leq \epsilon \int \abs{\nabla f}^qd\mu + \beta_q(\epsilon)\left(\int \abs{f}^{q/2}d\mu\right)^2
    \end{equation}
    for each $q \in [1, 2]$ and where 
    \[ \beta_q(\epsilon) \lesssim \exp(C\epsilon^{-p(\alpha+1)/(q(p-\alpha-1))}) \]
    for some $C > 0$ depending on $q$. 
\end{proposition}

\begin{remark}
    This $q$-super-Poincar\'e inequality is ``correct'' in the sense it is known by \cite[Lemma~4.11]{inglis2012spectral} that the inequality for $q < q'$ implies the inequality for $q'$ with a new constant depending $C$ and exponent in $\epsilon$ multiplied by $q/q'$. We do not claim at this stage however that the the growth $\beta_q$ is optimal.
\end{remark}

\begin{proof}
    If $q \in (1, 2]$ then replacing $f$ with $\abs{f}^q$ in \eqref{1spilocal} gives a local $q$-super-Poincar\'e inequality for the Lebesgue measure $d\xi$ with the same growth $\smash{\tilde{\beta}_q \lesssim \tilde{\beta}_1 \lesssim 1 + \delta^{-Q}}$ which, due to having only polynomial growth in $\delta^{-1}$, will not enter in the final asymptotic. (That being said, our examples are homogeneous spaces where $\smash{\tilde{\beta}_q} \lesssim 1 + \delta^{-Q/q}$.) The proof follows a standard decomposition of the space into a $d$-ball $B_R$ and its complement $B_R^c$. 
    
    Writing the measure $\mu$ as $d\mu = Z^{-1}e^{-V}d\xi$ for $V = N^p$ once again, on $B_R$, the measure is a log-bounded perturbation of Lebesgue measure and applying the $q$-super-Poincar\'e inequality to $\abs{f}e^{-V/q}$ yields a local $q$-super-Poincar\'e inequality for $\mu$, namely 
    \begin{align*}
        \int_{B_R} \abs{f}^qd\mu &\lesssim \int_{B_R} \abs{f}^qe^{-V}d\xi \vphantom{\epsilon \int \abs{\nabla f}^qd\mu + \epsilon \int_{B_R} \abs{f}^q\abs{\nabla V}^qd\mu + \tilde{\beta}_1(\epsilon) \sup_{B_R} e^V \left(\int_{B_R} \abs{f}^{q/2}d\mu\right)^2} \\
        &\leq \epsilon \int_{B_R} \abs{\nabla \abs{f}e^{-V/q}}^qd\xi + \beta_1(\epsilon)\left(\int_{B_R} \abs{\abs{f}e^{-V/q}}^{q/2}d\xi\right)^2 \vphantom{\epsilon \int \abs{\nabla f}^qd\mu + \epsilon \int_{B_R} \abs{f}^q\abs{\nabla V}^qd\mu + \tilde{\beta}_1(\epsilon) \sup_{B_R} e^V \left(\int_{B_R} \abs{f}^{q/2}d\mu\right)^2} \\
        &\lesssim \epsilon \int \abs{\nabla f}^qd\mu + \epsilon \int_{B_R} \abs{f}^q\abs{\nabla V}^qd\mu + \tilde{\beta}_1(\epsilon) \sup_{B_R} e^V \left(\int_{B_R} \abs{f}^{q/2}d\mu\right)^2 \\
        &\lesssim \epsilon \int \abs{\nabla f}^q + \tilde{\beta}_1(\epsilon) \sup_{B_R} e^{N^p} \left(\int \abs{f}^{q/2}d\mu\right)^2 \vphantom{\epsilon \int \abs{\nabla f}^qd\mu + \epsilon \int_{B_R} \abs{f}^q\abs{\nabla V}^qd\mu + \tilde{\beta}_1(\epsilon) \sup_{B_R} e^V \left(\int_{B_R} \abs{f}^{q/2}d\mu\right)^2} \\
        &\lesssim \epsilon \int \abs{\nabla f}^q + \tilde{\beta}_1(\epsilon) \exp(C R^p) \left(\int \abs{f}^{q/2}d\mu\right)^2 \vphantom{\epsilon \int \abs{\nabla f}^qd\mu + \epsilon \int_{B_R} \abs{f}^q\abs{\nabla V}^qd\mu + \tilde{\beta}_1(\epsilon) \sup_{B_R} e^V \left(\int_{B_R} \abs{f}^{q/2}d\mu\right)^2}
    \end{align*}
    for some $C > 0$ where at the final step we applied \eqref{comparison} and Lemma \ref{lem:ubound} since $\abs{\nabla V}^q \simeq U_q$. 
    
    On $B_R^c$ and for $q \in [1, 2)$, using the elementary lower bound $cx + x^{-s} \gtrsim c^{s/(s+1)}$ for $x, c, s > 0$, we have by Lemma \ref{lem:ubound} and the first part of Lemma \ref{lem:hardy}  
    \begin{equation}\label{uboundmerged}
        \int \abs{f}^q N^{q(p-\alpha-1)/(\alpha + 1)}d\mu \lesssim \int \abs{\nabla f}^qd\mu + \int \abs{f}^qd\mu
    \end{equation}
    and hence
    \begin{equation}\label{complement}
        \int_{B_R^c} \abs{f}^qd\mu \lesssim \frac{1}{R^\gamma} \int_{B_R^c} \abs{f}^qN^\gamma d\mu \lesssim \frac{1}{R^\gamma}\left(\int \abs{\nabla f}^q + \int \abs{f}^qd\mu\right)
    \end{equation}
    where $\gamma = q(p-\alpha-1)/(\alpha+1)$. Summing and taking $R = \epsilon^{-1/\gamma}$ we have \eqref{qspi} with the expected growth $\beta_q(\epsilon) = \smash{\tilde{\beta}}_1(\epsilon)\exp(C\epsilon^{-p/\gamma}) \lesssim \exp(C\epsilon^{-p(\alpha+1)/(q(p-\alpha-1))})$. 
    
    For $q \in [1, 2]$, in particular for $q = 2$, by the second part of Lemma \ref{lem:hardy} 
    \begin{align*}
        \int_{B_R^c} \abs{f}^qd\mu &\lesssim \left(\int_{B_R^c} \abs{f}^{q}\abs{x}^{q\alpha}d\mu\right)^{1/(\alpha+1)}\left(\int_{B_R^c} \abs{\nabla f}^qd\mu + \int_{B_R^c} \abs{f}^qd\mu\right)^{\alpha/(\alpha+1)} \vphantom{\left(\frac{1}{R^{\gamma(\alpha+1)}}\int_{B_R^c} \abs{f}^qU_qd\mu\right)^{1/(\alpha+1)}\left(\int \abs{\nabla f}^qd\mu + \int \abs{f}^qd\mu\right)^{\alpha/(\alpha+1)}} \\
        &\lesssim \left(\frac{1}{R^{\gamma(\alpha+1)}}\int_{B_R^c} \abs{f}^qU_qd\mu\right)^{1/(\alpha+1)}\left(\int \abs{\nabla f}^qd\mu + \int \abs{f}^qd\mu\right)^{\alpha/(\alpha+1)} \\
        &\lesssim \frac{1}{R^\gamma}\left(\int \abs{\nabla f}^q + \int \abs{f}^qd\mu\right) \vphantom{\left(\frac{1}{R^{\gamma(\alpha+1)}}\int_{B_R^c} \abs{f}^qU_qd\mu\right)^{1/(\alpha+1)}\left(\int \abs{\nabla f}^qd\mu + \int \abs{f}^qd\mu\right)^{\alpha/(\alpha+1)}}
    \end{align*}
    bringing us back to \eqref{complement}. The formality of the decomposition of the space $\bbG = B_R \sqcup B_R^c$ can be made rigorous by using a cutoff belonging to $W^{1, 2}(\mu)$, see, for instance, \cite{wang2000functional, cattiaux2009lyapunov, hebisch2010coercive}
\end{proof}

Finally, to prove the desired isoperimetric inequality, we specialise to the case $q = 1$. 

\begin{lemma}
    The $1$-super-Poincar\'e inequality \eqref{1spi} with growth 
    \[ \beta_1(\epsilon) \lesssim \exp(C\epsilon^{-1/\delta}) \quad \delta, C > 0 \] 
    implies the $1$-$F$-Sobolev inequality \eqref{1fsob} with 
    \[ F_1(x) = \log(1 + x)^\delta. \]
\end{lemma}

\begin{proof}
    The proof follows \cite[Theorem~3.2]{wang2000functional} except we replace 
    \begin{enumerate}
        \item $\mu(f^2) = 1$ with $\mu(f) = 1$, 
        \item $A_n = \{\delta^{n+1} > f^2 \geq \delta^n\}$ with $A_n = \{\delta^{n+1} > f \geq \delta^n\}$, and 
        \item $f_n = (f - \delta^{n/2}) \wedge (\delta^{(n+1)/2} - \delta^{n/2})$ with $f_n = (f - \delta^n) \wedge (\delta^{n+1} - \delta^n)$. 
    \end{enumerate}
    The proof follows in exactly the same way up until the second lower bound for $\mu(\abs{\nabla f}^2)$ where our analogue is 
    \[ \mu(\abs{\nabla f}) \geq \sum_{n=0}^\infty \eta(\delta^n)\mu(f \geq \delta^{n+1})(\delta^{n+1} - \delta^n) \]
    with $\eta$ the function $\xi$ defined in the statement of \cite[Theorem~3.2]{wang2000functional}, meaning we obtain the exact same function $F$ as in the case of the $2$-super-Poincar\'e inequality with growth $\beta_2(\epsilon) \lesssim \exp(C\epsilon^{-1/\delta})$ but with different constants, namely $c_1$ is no longer $(\delta^{(n+1)/2} - \delta^{n/2})^2/(\delta^n - \delta^{n-1})$ but instead 
    \[ c_1 = \frac{\delta^{n+1}- \delta^n}{\delta^n - \delta^{n-1}} = \delta. \qedhere \]
\end{proof}

\begin{proof}[Proof of Theorem \ref{thm1} for $p > \alpha + 1$]
    The $F$-Sobolev inequality implies by renormalisation a defective $1$-logarithmic${}^\theta$ Sobolev inequality of the form 
    \[ \int \abs{f}\log\left(1 + \frac{\abs{f}}{\int \abs{f}d\mu}\right)^\theta d\mu \lesssim \int \abs{\nabla f}d\mu + \int \abs{f}d\mu, \quad \theta = \frac{p - \alpha - 1}{p(\alpha + 1)} \] 
    which in the language of \cite{inglis2011u} is a defective $L^1\Phi$-entropy inequality for $\Phi(x) = x\log(1 = x)^\theta$. To satisfy the conditions of part (ii) of \cite[Theorem~4.5]{inglis2011u} we also need the Cheeger inequality 
    \[ \int \abs{f - \int fd\mu}d\mu \lesssim \int \abs{\nabla f}d\mu \]
    which follows by \cite[Theorem~2.6]{inglis2011u}, \eqref{1poincarelocal}, and the nondegenerate $U$-bound \eqref{uboundmerged}. The isoperimetric inequality follows from part (iii) of \cite[Theorem~4.5]{inglis2011u} with $q = 1/\theta$ and since $\cU_q$ defined in that paper is actually the isoperimetric profile of $d\nu_{q/(q-1)}$ in our notation. 
\end{proof}

\subsection{The case $p = \alpha + 1$}

We now prove the endpoint case. Since 
\[ \lim_{p \rightarrow \alpha + 1} \frac{p(\alpha + 1)}{p\alpha + (\alpha + 1)} = 1, \] 
we may expect $\mu$ for $p = \alpha + 1$ satisfies the linear isoperimetric inequality $I_\mu \gtrsim \cU_1$ of the double-sided exponential measure $\nu_1$, that is equivalently a Cheeger inequality; this would also parallel the euclidean setting where $\nu_r$ satisfies linear isoperimetry and the Cheeger inequality for $r = 1$ and superlinear isoperimetry and the super-Poincar\'e inequality for $r > 1$. However, the methods presented thus far apparently fail since the exponent on $\epsilon$ in the super-Poincar\'e inequality \eqref{qspi} blows up in the limit. Moreover, the Cheeger inequality of \cite[Theorem~2.6]{inglis2011u} used in the proof of Theorem \ref{thm1} again apparently fails since it requires the existence of a nondegenerate $U$-bound that diverges at infinity such as \eqref{uboundmerged} which is not provided in the critical case. (In fact, our arguments show the existence of such a $U$-bound implies superlinear isoperimetry.) 

Nevertheless, one can adapt the methods presented here to achieve a Cheeger inequality. Indeed, by \cite[Equation~2.24]{inglis2011u}, given the local $1$-Poincar\'e inequality \eqref{1poincarelocal} it is possible to prove a local Cheeger inequality (on a ball $B_R$) and this reduces the problem to showing
\begin{equation}\label{cheegerinfinity} 
    \int_{B_R^c} \abs{f - m} d\mu \lesssim \int \abs{\nabla f}d\mu
\end{equation}
for all $m \in \bbR$ holds, since $\int \abs{f - \int fd\mu}d\mu \leq 2\int \abs{f - m}d\mu$ and a choice of $m$ is made in the proof of \cite[Equation~2.24]{inglis2011u}.

To this end, we take $q = 1$ and return to \eqref{criticality1} 
\begin{equation}\label{criticality3}
    \int \abs{f}\abs{x}^\alpha d\mu = \int \abs{f}U_1d\mu \lesssim \int \abs{\nabla f}d\mu + \int \abs{f}U_1/N^pd\mu. 
\end{equation}
in the proof of Lemma \ref{lem:ubound} and also to \eqref{criticality2}
\begin{equation}\label{criticality4}
    \int \frac{\abs{f}}{\abs{x}}d\mu \lesssim \int \abs{\nabla f}d\mu + \int \abs{f}U_1d\mu.
\end{equation}
in the proof of Lemma \ref{lem:hardy}. 

Since we are only interested in estimating $\int_{B_R^c} \abs{f - m}d\mu$ in a neighbourhood of infinity, for $R$ sufficiently large we may absorb $U_1/N^p$ into the left hand side of \eqref{criticality3} which together with \eqref{criticality4} implies
\begin{equation}\label{criticality5}
    \int_{B_R^c} \abs{f}d\mu \lesssim \int_{B_R^c} \frac{\abs{f}}{\abs{x}}d\mu + \int_{B_R^c} \abs{f}\abs{x}^\alpha d\mu \lesssim \int_{B_R^c} \abs{\nabla f}d\mu. 
\end{equation}
Replacing $f$ by $f - m$ gives \eqref{cheegerinfinity} and thus $\mu$ at $p = \alpha + 1$ satisfies $I_\mu \gtrsim \cU_1$. The formality of the decomposition can be made rigorous by another approximation argument, see, for instance, \cite[pp.~4]{bakry2008rate}.
\section{Examples}

\subsection{Stratified Lie groups of step two}\label{S3.1}

A stratified Lie group $\bbG$ is a Lie group on $\bbR^n$ equipped with group law $\circ: \bbR^n \times \bbR^n \rightarrow \bbR^n$ and whose Lie algebra $\mf{g}$ of left invariant vector fields admits the decomposition $\mf{g} = \bigoplus_{i=0}^{\mf{r}-1} \mf{g}_i$ where the $\mf{g}_i$ are linear subspaces of $\mf{g}$ such that $\mf{g}_i = [\mf{g}_0, \mf{g}_{i-1}]$ for $1 \leq i \leq \mf{r}-1$ and $\mf{r} \in \bbZ_{\geq 1}$ the step of the group. We always work in a set of (exponential) coordinates on $\bbG$ in which we may write $\bbG \cong \otimes_{i=0}^{r-1} \bbR^{n_i}$ where $n_i = \dim(\mf{g}_i)$ and Lebesgue measure is the unique left invariant Haar measure up to constants. There is a canonical basis $\{X_1, \cdots, X_\ell\}$ for $\mf{g}_0$ whose components form the subgradient $\nabla_\bbG = (X_1, \cdots, X_\ell)$ and the sublaplacian $\Delta_\bbG = \nabla_\bbG \cdot \nabla_\bbG = \sum_{i=1}^\ell X_i^2$, analogising the euclidean gradient and laplacian respectively. The subgradient induces a natural metric on $\bbG$, called the Carnot-Carath\'eodory distance $d = d_\bbG$, while the sublaplacian induces on groups of homogeneous dimension $Q = Q(\bbG) = \sum_{i=0}^{r-1} (i+1)\dim(\mf{g}_i) \geq 3$ a homogeneous norm, called the Kor\'anyi-Folland gauge $N$, through the distributional identity $\Delta_\bbG N^{2-Q} = \delta_0$, that is $N^{2-Q}$ is the fundamental solution of $\Delta_\bbG$. 

When $\bbG$ has step $\mf{r} = 2$, it is known by \cite[Theorem~3.2.2]{bonfiglioli2007stratified} that the group law can be characterised by the group law
\begin{equation}\label{grouplawstep2}
    (x, t) \circ (\xi, \tau) = \left(x + \xi, t_1 + \tau_1 + \frac{1}{2} \inner{B^{(1)}x, \xi}, \cdots, t_n + \tau_n + \frac{1}{2} \inner{B^{(n)}x, \xi}\right )
\end{equation}
on $\bbR^n_x \times \bbR^m_t$ where $B^{(1)}, \cdots, B^{(m)}$ is a family of $n \times n$ linearly independent skew-symmetric matrices. We say that $\bbG$ is a $H$-type group if the $B^{(j)}$ are all orthogonal and anticommute. This class is special since the Kor\'anyi-Folland gauge associated to a $H$-type group is given by the explicit formula
\begin{equation}\label{Nkappa}  
    N_\kappa(x, t) = (\abs{x}^4 + \kappa \abs{t}^2)^{1/4}, \quad \kappa > 0, 
\end{equation}
for $\kappa = 16$ and is called the Kaplan norm after \cite[Theorem~2]{kaplan1980fundamental}. The class of $H$-type groups contains the Heisenberg group $\bbH^1$ together with its higher order generalisations, and is itself generalised by the class of M\'etivier groups whose $B^{(j)}$ satisfy the weaker property that every nonzero linear combination of the $B^{(j)}$ is nonsingular. Note by \cite[Remark~3.6.3]{bonfiglioli2007stratified} that the first layer $\bbR^n_x$ of a M\'etivier group has even dimension $n \geq 2$. For more details on stratified Lie groups (of step two) we refer the reader to \cite[\S1,~\S3,~\S18]{bonfiglioli2007stratified}. 

We now verify the conditions of Theorem \ref{thm1}. On a $H$-type group $\bbG \cong \bbR^n_x \times \bbR^m_t$, the estimates \eqref{estimates} on $N$ are exact by \cite[Proposition~2.7~and~pp.~20]{inglis2012spectral} and hold with $\alpha = 1$ and $\abs{x}$ the euclidean norm on $\bbR^n_x$ which has (even) dimension $n \geq 2$. When $\bbG$ is a general step-two group the first layer has again dimension $n \geq 2$ otherwise $\bbG$ is trivial, and functions of the form \eqref{Nkappa} were proven in \cite[Lemma~2]{bou2021coercive} to satisfy \eqref{estimates} with $\alpha = 1$. The subgradient and sublaplacian coincide with their euclidean counterparts when acting on functions solely of $x \in \bbR^n_x$, and therefore \eqref{laplacian} is automatic. The first part of \eqref{comparison} follows by definition of $N_\kappa$ and from \cite[Proposition~5.1.4]{bonfiglioli2007stratified}. Finally, as mentioned earlier, the $L^1$-Sobolev inequality \eqref{1spilocal} with $Q = Q(\bbG)$ and the local $1$-Poincar\'e inequality \eqref{1poincarelocal} follow from \cite[Theorem~IV.7.1]{varopoulos1991analysis} and \cite[Theorem~2.1]{jerison1986poincare} respectively. 

\begin{corollary}
    Let $\bbG \cong \bbR^n_x \times \bbR^m_t$ be a step-two stratified Lie group equipped with its Carnot-Carath\'eodory distance $d$ and let $N$ be given by $N_\kappa(x, t) = (\abs{x}^4 + \kappa \abs{t}^2)^{1/4}$, $\kappa > 0$. Then the measure $d\mu = Z^{-1}e^{-N^p}d\xi$, $p \geq 2$, satisfies the isoperimetric inequality 
    \[ I_\mu \gtrsim \cU_{2p/(p+2)}. \]
\end{corollary}

This completes a series of works starting from \cite[Theorem~4.19]{inglis2012spectral} wherein it was proved that $\mu$ on a $H$-type group $\bbG$ satisfied the Poincar\'e inequality for $p \geq 2$ and conjectured to satisfy a $2$-super-Poincar\'e inequality for $p > 2$, more precisely that, equivalently, the underlying operator had empty essential spectrum. This was answered affirmatively by \cite[Theorem~4.4]{bruno2017weighted} through spectral theory arguments based on \cite[Theorem~3]{simon2008schrodinger} for the larger class of M\'etivier groups. The equivalent statement for step-two groups from the probabilistic perspective is the content of the $q$-super-Poincar\'e inequality \eqref{qspi} at $q = 2$, see also \cite[Corollary~12]{bou2021coercive} wherein a $2$-$F$-Sobolev inequality \eqref{2fsob0} with $F(x) = x\log(1 + x)^\theta$ was proved with suboptimal $\theta$ and conditions on $p$. 

Furthermore, the isoperimetric inequality is optimal in the sense $2p/(p+2)$ cannot be replaced with anything strictly larger. Indeed, it implies a stronger $1$-super-Poincar\'e inequality (meaning the exponent $2p/(p-2)$ of \eqref{qspi} at $q = 1$ can be improved) which in turn implies a stronger $2$-super-Poincar\'e inequality. However, this would be impossible by an argument of \cite{hebisch2010coercive} which also appears in \cite[Theorem~4.16]{inglis2012spectral}, the latter of which we follow for simplicity since the measure there is the same as the measure here. With $\varphi$ the function defined by \cite[Equation~4.14]{inglis2012spectral}, it satisfies $\int \varphi^2d\mu \gtrsim r^Q\exp(-t^pN^p(\xi_0))$ and $\int \abs{\nabla \varphi}^2d\mu \lesssim r^{Q-2} \exp(-t^pN^p(\xi_0))$ for some $\xi_0 = (0, s) \in \bbG^*$. In their notation, $r = t^{1-p/2}$ where $t > 0$ is a parameter which can be taken arbitrarily large. Since $\beta_2(\epsilon)(\int \abs{\varphi}d\mu)^2 \lesssim \beta_2(\epsilon)r^{2Q}\exp(-2t^pN^p(\xi_0))$, we see that \eqref{qspi} reads
\[ 1 \lesssim \epsilon r^{-2} + r^Q\beta_2(\epsilon)\exp(-t^pN^p(\xi_0)) \lesssim \epsilon t^{p-2} + t^{-Q(p/2-1)}\exp(C\epsilon^{-p/(p-2)} - t^pN^p(\xi_0)) \] 
after division by $r^Q\exp(-t^pN^p(\xi_0))$. Thus if the growth $\beta_2(\epsilon)$ enjoyed an exponent smaller than $p/(p-2)$, for instance say $\beta_2(\epsilon) \lesssim \exp(C\epsilon^{-p/(p-2)+\delta})$ for some $C > 0$ and $0 < \delta \ll 1$, taking $t = \epsilon^{-\gamma}$ for any $\frac{1}{p-2} - \frac{\delta}{p} < \gamma < \frac{1}{p-2}$ yields a contradiction as $\epsilon \rightarrow 0^+$. In any case, the optimality of this isoperimetric inequality is consistent with the fact $\mu$ cannot have gaussian isoperimetry otherwise it would satisfy the $2$-logarithmic Sobolev inequality by an argument of \cite{ledoux1994simple} which in turn contradicts \cite[Theorem~6.3]{hebisch2010coercive}. 

Lastly, although the fact the $2$-super-Poincar\'e inequality implies the correct exponent on $\epsilon$ in the $1$-super-Poincar\'e inequality may suggest the possibility of extracting isoperimetric content directly from the $2$-super-Poincar\'e inequality, it is not known to us how (or even if this is possible). Indeed, the isoperimetric content of \cite[Theorem~3.4]{wang2000functional} is extracted under Bakry-\'Emery type curvature lower bounds (through a reverse Poincar\'e inequality) which are not available in the subelliptic setting. A weaker but sufficient assumption would be the generalised curvature condition of \cite{baudoin2012log}, but since our measures have supergaussian tails but do not satisfy the $2$-logarithmic Sobolev inequality, they cannot satisfy the generalised curvature condition since this would contradict \cite[Theorem~1.4]{baudoin2012log}. 

Before we conclude our discussion of step-two groups, let us note the works \cite{hebisch2010coercive, inglis2011u, inglis2012spectral, antonelli2024sharp} which considered functional inequalities for measures of the form \eqref{measure} on $H$-type groups but with $N$ replaced with the Carnot-Carath\'eodory distance $d$ itself; in fact \cite{inglis2011u, inglis2012spectral} also studied isoperimetric and super-Poincar\'e inequalities respectively. The situation here is somewhat easier in the sense $d$ is better behaved and $d\mu = Z^{-1}e^{-d^p}d\xi$ satisfies the isoperimetric inequality $I_\mu \gtrsim \cU_p$ according to \cite[Theorem~5.6]{inglis2011u}, but also rather delicate in the sense $d$ is neither smooth nor explicit which makes it difficult to verify the essential ingredients replacing \eqref{estimates}, namely the eikonal equation $\abs{\nabla d} = 1$ for Carnot-Carath\'eodory spaces \cite[Theorem~3.1]{monti2001surface}, and the laplacian comparison $\Delta d \lesssim 1/d$ for $H$-type groups \cite[Theorem~6.1]{hebisch2010coercive}. These results generalise to arbitrary step-two groups where the laplacian comparison follows from \cite[Corollary~4.17]{cavalletti2020new}, since such groups are measure contraction spaces \cite[Theorem~3]{badreddine2020measure}.

\subsection{Grushin and Heisenberg-Greiner operators}

While the following operators are not in general the subgradient and sublaplacian on a stratified Lie group, they nonetheless satisfy similar estimates (indeed, the Grushin and Heisenberg-Greiner operators generalise their euclidean and Heisenberg counterparts at a certain choice of parameter) and therefore the conditions of Theorem \ref{thm1}. 

The Grushin subgradient, following \cite{d'ambrosio2004hardy}, is the operator 
\[ \nabla_\eta = (\nabla_x, \abs{x}^\eta \nabla_y), \quad \eta > 0, \] 
acting on $\bbR^n_x \times \bbR^m_y$. We assume moreover $n \geq 2$. The Grushin sublaplacian $\Delta_\eta = \nabla_\eta \cdot \nabla_\eta = \Delta_x + \abs{x}^{2\eta}\Delta_y$ generalises the euclidean laplacian at $\eta = 0$, is homogeneous of order $2$ with respect to the anisotropic dilations $\delta_\lambda(x, y) = (\lambda x, \lambda^{1 + \eta}y)$, and is hypoelliptic for $\eta > 0$  due to \cite[Theorem~1.2]{gruvsin1970class} and known already for even $\eta \in 2\bbZ_{\geq 1}$ by H\"ormander's rank condition. Its ``Kor\'anyi-Folland gauge", in the sense of the fundamental solution of $\Delta_\eta$, is 
\begin{equation}\label{Neta} 
    N_\eta(x, y) = (\abs{x}^{2(1 + \eta)} + (1 + \eta)^2\abs{y}^2)^{1/(2(1 + \eta)}, 
\end{equation}
and satisfies \eqref{estimates} with $\alpha = \eta$ and $\abs{x}$ the euclidean norm of $\bbR^n_x$. The action of $\nabla_\eta$ and $\Delta_\eta$ is euclidean on functions depending only on $x$ as before, giving us \eqref{laplacian}, while \eqref{comparison} follows by previous arguments. The $L^1$-Sobolev (with $Q$ the homogeneous dimension $Q(n, m, \eta) = n + (1 + \eta)m$, in analogue with the group setting) and local $1$-Poincar\'e inequalities are due to \cite[Theorem~1.1]{capogna1994geometric} and \cite[Theorem~2]{franchi1994weighted} respectively, which provides us with a full family of probability measures with arbitrarily fast decay of tails satisfying an isoperimetric inequality arbitrarily close to linear. 

\begin{corollary}\label{cor:grushin}
    Let $\bbR^n_x \times \bbR^m_y$ be the Grushin space with $n \geq 2$, equipped with its Carnot-Carath\'eodory distance $d$ induced by $\nabla_\eta$, $\eta > 0$, and let $N_\eta$ be given by \eqref{Neta}. Then the measure $d\mu = Z^{-1}e^{-N_\eta^p}d\xi$, $p \geq 1 + \eta$, satisfies the isoperimetric inequality 
    \[ I_\mu \gtrsim \cU_{p(1 + \eta)/(p\eta + 1 + \eta)}. \]
\end{corollary}

As before, we can also consider the optimality of Theorem \ref{thm1}, at least for nonnegative integers $\eta \in \bbZ_{\geq 1}$. Briefly, we can modify the proof of \cite[Lemma~6.3]{hebisch2010coercive} to conclude $\abs{N_\eta(\xi) - N_\eta(\xi_0)} \lesssim d^{1 + \eta}(\xi, \xi_0)$ at points $\xi_0 = (0, y) \in (\bbR^n_x \times \bbR^m_y)^*$ since it can be explicitly verified $X_j \cdots X_1N_\eta(\xi_0) = 0$ for $X_i \in \{\partial_{x_1}, \cdots, \partial_{x_n}, \abs{x}^\eta \partial_{y_1}, \cdots, \abs{x}^\eta \partial_{y_m}\}$ any of the vector fields in $\nabla_\eta$ and $1 \leq i, j \leq \eta$. Our previous proof can be replicated with the minor replacement of the $r$-scale by $r = t^{1-p/(1 + \eta)}$ so that \eqref{qspi} reads 
\[ 1 \lesssim \epsilon t^{2(p/(1 + \eta) - 1)} + t^{-Q(p/(1 + \eta)-1)}\exp(C\epsilon^{-p(1 + \eta)/(2(p-1-\eta))} - t^pN^p(\xi_0)) \]
and the previous argument shows the exponent on the $\epsilon$ inside the exponential is optimal.

Lastly, note if $\eta \in (0, 1)$ then $\mu$ has gaussian isoperimetry for $p \geq 2(1 + \eta)/(1 - \eta)$ and in fact, by the fourth and fifth parts of \cite[Theorem~4.5]{inglis2011u}, satisfy the $q$-logarithmic Sobolev inequality (the $L^q$-analogue of \eqref{1fsob} with $F(x) = \log(1 + x)$) as well as the Bobkov type functional isoperimetric inequalities 
\[ \cU_2\left(\int fd\mu\right) \leq \int \sqrt{\cU_2^2(f) + C\abs{\nabla f}^2}d\mu \] 
of \cite{bobkov1996functional, bobkov1997isoperimetric}, see also \cite[\S16]{bobkov2005entropy}. 

Similarly, the Heisenberg-Greiner subgradient, following \cite{d'ambrosio2005hardy}, is the operator 
\[ \nabla_\zeta = (X_\zeta^1, \cdots, X_\zeta^n, Y_\zeta^1, \cdots, Y_\zeta^n), \quad \zeta \geq 1, \]
acting on $(\bbR^n_x \times \bbR^n_y) \times \bbR^1_t$ and where 
\begin{equation}\label{greinerfields}
    X_\zeta^i = \partial_{x_i} + 2\zeta y_i\abs{r}^{2\zeta-2}\partial_t, \quad Y_\zeta^i = \partial_{y_i} - 2\zeta x_i\abs{r}^{2\zeta-2}\partial_t 
\end{equation}
for each $i = 1, \cdots, n$ and where $r = (x, y) \in \bbR^n_x \times \bbR^n_y$. The Heisenberg-Greiner sublaplacian $\Delta_\zeta = \nabla_\zeta \cdot \nabla_\zeta = \sum_{i=1}^n X_i^2 + Y_i^2$ generalises the Heisenberg laplacian at $\zeta = 1$, is homogeneous of order $2$ with respect to the anisotropic dilations $\delta_\lambda(x, y, t) = (\lambda x, \lambda y, \lambda^{2\zeta} t)$, and is hypoelliptic for each nonnegative integer $\zeta \in \bbZ_{\geq 1}$. Again, we can verify the conditions of Theorem \ref{thm1} for
\begin{equation}\label{Nzeta}
    N_\zeta(x, y, t) = (\abs{r}^{4\zeta} + t^2)^{1/(4\zeta)}
\end{equation}
with $\alpha = 2\zeta - 1$ and $(r, t)$ playing the role of $(x, x')$. That the vector fields for $\zeta \in \bbZ_{\geq 1}$ satisfy H\"ormander's rank condition implies the $L^1$-Sobolev (with $Q = Q(n, \zeta) = 2n + 2\zeta$) and local $1$-Poincar\'e inequality follow from \cite[Theorem~1.1]{capogna1994geometric} and \cite[Theorem~2.1]{jerison1986poincare} respectively. Optimality follows from previous arguments. 

\begin{corollary}
    Let $(\bbR^n_x \times \bbR^n_y) \times \bbR^1_t$ be the Heisenberg-Greiner space equipped with its Carnot-Carath\'eodory distance $d$ induced by $\nabla_\zeta$, $\zeta \in \bbZ_{\geq 1}$, and let $N_\zeta$ be given by \eqref{Nzeta}. Then the measure $d\mu = Z^{-1}e^{-N_\zeta^p}d\xi$, $p \geq 2\zeta$, satisfies the isoperimetric inequality 
    \[ I_\mu \gtrsim \cU_{2p\zeta/(p(2\zeta-1)+2\zeta)}. \]
\end{corollary}

\subsection{The one-dimensional case $n_1 = 1$}

\subsubsection{An almost $L^1$-Hardy inequality}

Recalling the Grushin example assumed the dimension $n$ of the first layer, playing the role of $n_1$ in the statement of Theorem \ref{thm1}, satisfied $n \geq 2$, we note this assumption avoids the technical issue of the absence of the $L^1$-Hardy inequality in dimension $1$. It is automatic in the setting of step-two groups, since the group is trivial if the first layer has dimension $1$, and in the Heisenberg-Greiner setting, since the dimension of the singularity $\{r = 0\}$ along which the $U$-bound degenerates is $2n \geq 2$. 

If we look at the proofs, we observe this issue can be sidestepped if one could obtain the nondegenerate $U$-bound \eqref{uboundmerged} without the $L^1$-Hardy inequality, which provides the $1$-super-Poincar\'e and Cheeger inequalities. We can almost resolve the one-dimensional case because there exists an almost $L^1$-Hardy inequality. Indeed, as before with $V = N^p$ and $0 < \delta < 1$ small but fixed,
\[ \delta(1 + \delta)\int \frac{\abs{f}}{\abs{x}^{1-\delta}}d\mu = \int \abs{f}e^{-V} \nabla \cdot \nabla \abs{x}^{1+\delta}d\xi \]
and 
\begin{align*}
    \int \abs{f}e^{-V} \nabla \cdot \nabla \abs{x}^{1+\delta}d\xi &= -\int \nabla \abs{f} \cdot \nabla \abs{x}^{1+\delta} d\mu + \int \abs{f} \nabla V \cdot \nabla \abs{x}^{1+\delta}d\mu \\ 
    &\lesssim (1 + \delta)\left(\int \abs{\nabla f}\abs{x}^\delta d\mu + \int \abs{f}\abs{x}^\delta U_1 d\mu\right). 
\end{align*}
so that 
\begin{equation}\label{almosthardy}
    \int \frac{\abs{f}}{\abs{x}^{1-\delta}}d\mu \lesssim \frac{R_0^\delta}{\delta}\int \abs{\nabla f}d\mu + \frac{R_0^\delta}{\delta} \int \abs{f}U_1d\mu + \frac{1}{R_0^{1-\delta}}\int \abs{f}d\mu 
\end{equation}
by decomposing the space into the $\abs{x}$-ball of radius $R_0 > 0$ and its complement. 

At the endpoint case $p = \alpha + 1$, on the $d$-ball $B_R^c$ of radius $R$ sufficiently large, we recover the analogue 
\[ \int_{B_R^c} \abs{f}d\mu \lesssim \int_{B_R^c} \frac{\abs{f}}{\abs{x}^{1-\delta}}d\mu + \int_{B_R^c} \abs{f}\abs{x}^\alpha d\mu \lesssim \int_{B_R^c} \abs{\nabla f}d\mu + \epsilon \int \abs{f}d\mu \]
of \eqref{criticality5} by \eqref{criticality3} and taking $R_0$ sufficiently large (so that $1/R_0^{1-\delta}$ plays the role of $\epsilon$) in \eqref{almosthardy}. Otherwise, if $p > \alpha + 1$, then \eqref{almosthardy} simply replaces the Hardy inequality \eqref{hardy1} everywhere so that \eqref{uboundmerged} holds with corrected exponent on $N$, namely
\[ \int \abs{f}^qN^{q(p-\alpha-1)(1-\delta)/(\alpha+1-\delta)}d\mu \lesssim \int \abs{\nabla f}^qd\mu + \int \abs{f}^qd\mu. \] 
This leads to the following family of almost isoperimetric inequalities interpolating between linear isoperimetry at $\delta = 1$ and the expected isoperimetry at $\delta = 0$.

\begin{proposition}\label{thm3}
    Assume as in Theorem \ref{thm1} except that $n_1 = 1$. If $p = \alpha + 1$ then $I_\mu \gtrsim \cU_1$, and otherwise if $p > \alpha + 1$ then $I_\mu$ satisfies 
    \begin{equation}\label{almostisoperimetry}
        I_\mu \gtrsim_\delta \cU_r, \quad r = \frac{(\alpha+1-\delta)p}{\alpha+1-\delta + \alpha(p-\delta)} \in \left[1, \frac{(\alpha+1)p}{\alpha+1+\alpha p}\right)
    \end{equation}
    for each $\delta \in (0, 1]$ and where $\gtrsim_\delta$ indicates the constant depends on $\delta$.
\end{proposition}

\subsubsection{An $L^1$-Caffarelli-Kohn-Nirenberg inequality}

Alternatively, the reader may have noticed it suffices also to have the Caffarelli-Kohn-Nirenberg inequality \eqref{uncertainty1} at $q = 1$. Indeed, it also provides the $1$-super-Poincar\'e inequality (see the second part of the proof of Proposition \ref{prop:spi}), and it can be linearised to imply the Cheeger inequality (see the proof below). However, to our best knowledge, there is no proof, not passing through the $L^1$-Hardy inequality, of the $L^1$-Caffarelli-Kohn-Nirenberg inequality on a stratified Lie group with respect to Lebesgue measure and where the homogeneous seminorm $\abs{x}$ is replaced with an arbitrary homogeneous norm, and an inequality for $\mu$ appears even harder given the fact the original proof by \cite{caffarelli1984first} depended on the positive homogeneity of Lebesgue measure. For some $L^q$-results, $q > 1$, see, for instance, \cite{ruzhansky2017caffarelli, ruzhansky2018extended, kassymov2022reverse, ruzhansky2024hardy} or the book \cite{ruzhansky2019hardy}. While we will not address the existence of a general $L^1$-Caffarelli-Kohn-Nirenberg inequality for (generic) probability measures $\mu$, we will nevertheless be able to show there exists a property of the Grushin fields which allows one to prove \eqref{uncertainty1} at $q = 1$. Moreover, this property is somewhat generic in the sense it persists for a large class of stratified Lie groups.

\begin{theorem}\label{thm2}
    Assume as in Theorem \ref{thm1} except that $n_1 = 1$ so that $\xi = (x_1, x') \in \bbR^1 \times \bbR^{n-1}$. If $\nabla$ contains $\partial_{x_1}$, then the conclusion of Theorem \ref{thm1} continues to hold. 
\end{theorem}

\begin{proof}
    Since $\nabla$ contains the euclidean field $\partial_{x_1}$, we can apply the \emph{euclidean} Caffarelli-Kohn-Nirenberg inequality \cite[Theorem~1]{caffarelli1984first} on $\bbR^1_{x_1}$ with respect to $\partial_{x_1}$ and start with the one-dimensional inequality
    \[ \int_{\bbR^1_{x_1}} \abs{f}dx_1 \lesssim \left(\int_{\bbR^1_{x_1}} \abs{f}\abs{x_1}^{\alpha}dx_1\right)^{1/(\alpha+1)}\left(\int_{\bbR^1_{x_1}} \abs{\partial_{x_1}f}dx_1\right)^{\alpha/(\alpha+1)}. \]
    We can then integrate with respect to the remaining variables $x' \in \bbR^{n-1}$ so that 
    \begin{equation}\label{1dimckn}
        \begin{aligned}
            \int \abs{f}d\xi &\lesssim \int_{\bbR^{n-1}_{x'}} \left[\left(\int_{\bbR^1_{x_1}} \abs{f}\abs{x_1}^\alpha dx_1\right)^{1/(\alpha+1)}\left(\int_{\bbR^1_{x_1}} \abs{\partial_{x_1}f}dx_1\right)^{\alpha/(\alpha+1)}\right]dx' \\
            &\lesssim \left(\int \abs{f}\abs{x_1}^\alpha d\xi\right)^{1/(\alpha+1)} \left(\int \abs{\partial_{x_1}f}d\xi\right)^{\alpha/(\alpha+1)} \vphantom{\int_{\bbR^{n-1}_{x'}} \left[\left(\int_{\bbR^1_{x_1}} \abs{f}\abs{x_1}^\alpha dx_1\right)^{1/(\alpha+1)}\left(\int_{\bbR^1_{x_1}} \abs{\partial_{x_1}f}dx_1\right)^{\alpha/(\alpha+1)}\right]dx'} \\
            &\lesssim \left(\int \abs{f}\abs{x_1}^\alpha d\xi\right)^{1/(\alpha+1)} \left(\int \abs{\nabla f}d\xi\right)^{\alpha/(\alpha+1)} \vphantom{\int_{\bbR^{n-1}_{x'}} \left[\left(\int_{\bbR^1_{x_1}} \abs{f}\abs{x_1}^\alpha dx_1\right)^{1/(\alpha+1)}\left(\int_{\bbR^1_{x_1}} \abs{\partial_{x_1}f}dx_1\right)^{\alpha/(\alpha+1)}\right]dx'}
        \end{aligned}        
    \end{equation} 
    by H\"older's inequality and since $\abs{\partial_{x_1}f} \leq \abs{\nabla f}$. This is the $L^1$-Caffarelli-Kohn-Nirenberg inequality for $\nabla$ we seek, but unfortunately with respect to Lebesgue measure. To obtain the same inequality with respect to $\mu$, write $d\mu = Z^{-1}e^{-V}d\xi$ for $V = N^p$, and then replace directly $f$ with $fe^{-V}$. Then by \eqref{estimates}, we have 
    \begin{equation}\label{insertion} 
        \abs{\nabla fe^{-V}} \lesssim \abs{\nabla f}e^{-V} + \abs{f}N^{p-1}\abs{\nabla N}e^{-V} \lesssim \abs{\nabla f}e^{-V} + \abs{f}U_1e^{-V}
    \end{equation}
    so that by the $U$-bound \eqref{ubound1}, we arrive at \eqref{uncertainty1} with $q = 1$. 

    If $p > \alpha + 1$, we already showed how to obtain the super-Poincar\'e inequality in Proposition \ref{prop:spi}, so all which remains is the Cheeger inequality. For each $p \geq \alpha + 1$, we see that 
    \begin{align*}
        \int_{B_R^c} \abs{f}d\mu &\lesssim \left(\int_{B_R^c} \abs{f}\abs{x_1}^\alpha d\mu\right)^{1/(\alpha+1)}\left(\int_{B_R^c} \abs{\nabla f}d\mu + \int_{B_R^c} \abs{f}d\mu\right)^{\alpha/(\alpha+1)} \\
        &\lesssim \delta \int_{B_R^c} \abs{\nabla f}d\mu + \delta \int_{B_R^c} \abs{f}d\mu \int_{B_R^c} \abs{f}\abs{x_1}^\alpha d\mu \vphantom{\lesssim \left(\int_{B_R^c} \abs{\nabla f}d\mu\right)^{\alpha/(\alpha+1)}\left(\int_{B_R^c} \abs{f}\abs{x}^\alpha d\mu\right)^{1/(\alpha+1)}} \\ 
        &\lesssim \delta \int_{B_R^c} \abs{\nabla f}d\mu + \delta \int_{B_R^c} \abs{f}d\mu + \frac{1}{R^{p-\alpha-1}} \int_{B_R^c} \abs{f}U_1 d\mu \vphantom{\lesssim \left(\int_{B_R^c} \abs{\nabla f}d\mu\right)^{\alpha/(\alpha+1)}\left(\int_{B_R^c} \abs{f}\abs{x}^\alpha d\mu\right)^{1/(\alpha+1)}}
    \end{align*}
    we have directly \eqref{criticality5} after taking into account \eqref{criticality3} for $R$ sufficiently large and $\delta$ small enough to absorb the $L^1$-defect into the left hand side. Given the $1$-super-Poincar\'e inequality and this Cheeger inequality, the rest follows as before. 
\end{proof}

\begin{corollary}
    The conclusion of Corollary \ref{cor:grushin} continues to hold in the case $n = 1$. 
\end{corollary}

\begin{remark}
    In fact, in the Grushin setting the previous argument can go through in the general case $n_1 \geq 1$ since $\nabla_\eta$ contains the full euclidean gradient $\nabla_x$. 
\end{remark}

Lastly, we show an adaptation of the previous idea allows us to study a similar situation which plays out again in the setting of stratified Lie groups. Our starting point is \cite[Theorem~4.3]{chatzakou2023q} wherein a $q$-Poincar\'e inequality was proved for a probability measure of the form \eqref{measure} on a group $\bbG \cong (\bbR^{n+1}, \circ)$ of step $n \geq 3$ with filiform Lie algebra, meaning $\mf{g}$ is generated by two vector fields $X_1$ and $X_2$ so that thus, in our earlier notation, $\mf{g}_i$ for $1 \leq i \leq n-1$ is spanned by a single vector field. The main objects were 
\begin{enumerate}
    \item 
    \[ \norm{x}^n = \sum_{j=2}^n \left(\abs{x_1}^{(n+1)/2} + \abs{x_2}^{(n+1)/2} + \abs{x_j}^{(n+1)/(2(j-1))}\right)^{2n/(n+1)}, \]
    \item $N(x) = (\norm{x}^n + \abs{x_{n+1}})^{1/n}$, 
    \item $X_1 = \partial_{x_1}$, and 
    \item $X_2 = \partial_{x_2} + x_1 \partial_{x_3} + \frac{x_1^2}{2}\partial_{x_4} + \cdots + \frac{x_1^{n-1}}{(n-1)!}\partial_{x_{n+1}}$. 
\end{enumerate}
What is interesting is that the proof of the $U$-bound does not involve taking the full subgradient $\nabla = \nabla_\bbG = (X_1, X_2)$ but only the $X_1$-derivative and exploits the fact 
\begin{equation}\label{filiformestimate}
    X_1 N \cdot X_1 \abs{x_1} \gtrsim \abs{x_1}^{(n-1)/2}\norm{x}^{(n-1)/2}N^{-(n-1)} \gtrsim \abs{x_1}^{n-1}N^{-(n-1)}
\end{equation}
so that if $V = N^p$ as before then integrating by parts shows 
\begin{align*}
    0 = -\int \abs{f}X_1^2\abs{x_1}d\mu &= \int X_1(\abs{f}e^{-V}) \cdot X_1\abs{x_1}d\xi \\
    &= \int X_1\abs{f} \cdot X_1\abs{x_1}d\mu - \int \abs{f}N^{p-1}X_1N \cdot X_1\abs{x_1}d\mu \\
    &\lesssim \int \abs{\nabla_\bbG f}d\mu - \int \abs{f}\abs{x_1}^{n-1}N^{p-n}d\mu
\end{align*} 
since $X_1\abs{f} \leq \abs{X_1f} \leq \abs{\nabla_\bbG f}$ and $X_1^2\abs{x_1} = 0$. This rearranges into a $U$-bound 
\begin{equation}\label{filiformubound}
    \int \abs{f}\abs{x_1}^{n-1}N^{p-n}d\mu \lesssim \int \abs{f}N^{p-1}X_1N \cdot X_1\abs{x_1}d\mu \lesssim \int \abs{\nabla_\bbG f}d\mu
\end{equation}
so that, together with an almost $L^1$-Hardy inequality \eqref{almosthardy} with singularity along $\{\abs{x_1} = 0\}$, if $p = \alpha + 1 = n$ then we obtain linear isoperimetry $I_\mu \gtrsim \cU_1$, and otherwise if $p > n$ then we have the family \eqref{almostisoperimetry} of isoperimetric inequalities, with respect to the Carnot-Carath\'eodory distance $d$ on $\bbG$ induced by $\nabla_\bbG$. To achieve the expected isoperimetric inequality corresponding to the endpoint case $\delta = 0$ in \eqref{almostisoperimetry} however, we would like to prove the $L^1$-Caffarelli-Kohn-Nirenberg inequality as was done in the Grushin setting. The setup is exactly the same; $\nabla_\bbG$ contains $X_1 = \partial_{x_1}$, except the $U$-bound was not proved with respect to $\nabla_\bbG$ but $X_1$ and we do not know if $\abs{\nabla_\bbG N} \lesssim X_1N \cdot X_1\abs{x_1}$. However, what is true is that 
\[ \abs{X_1N} \lesssim \abs{x_1}^{(n-1)/2} \norm{x}^{(n-1)/2} N^{-(n-1)} \lesssim X_1N \cdot X_1\abs{x_1} \]
which follows from explicit computation and \eqref{filiformestimate}, so we can return to the second inequality of \eqref{1dimckn} with $\alpha = n-1$ and replace $f$ with $fe^{-V}$ before replacing $X_1$ with $\nabla_\bbG$, so that 
\begin{align*}
    \abs{X_1fe^{-V}} \lesssim \abs{X_1f}e^{-V} + \abs{f}N^{p-1}\abs{X_1N}e^{-V} &\lesssim \abs{X_1f}e^{-V} + \abs{f}X_1N \cdot X_1\abs{x_1}e^{-V}
\end{align*}
becomes the analogue of \eqref{insertion}, giving together with \eqref{filiformubound}
\begin{align*}
    \int \abs{f}d\mu &\lesssim \left(\int \abs{f}\abs{x}^{n-1}d\mu\right)^{1/n}\left(\int \abs{\nabla_\bbG f}d\mu + \int \abs{f}d\mu\right)^{(n-1)/n}. 
\end{align*}

\begin{proposition}
    Let $\bbG \cong (\bbR^{n+1}, \circ)$ be a stratified Lie group with filiform Lie algebra of step $n \geq 3$ equipped with its Carnot-Carath\'eodory distance $d$ induced by $\nabla_\bbG$ and the homogeneous norm $N$ as defined in \textup{\cite[pp.~19]{chatzakou2023q}}. Then the measure $d\mu = Z^{-1}e^{-N^p}d\xi$, $p \geq n$, satisfies the isoperimetric inequality 
    \[ I_\mu \gtrsim \cU_{pn/(n+p(n-1))}. \]    
\end{proposition}

\begin{remark}
    The same arguments hold for those measures considered in and satisfying the conditions of \cite[Lemma~4.6]{chatzakou2022poincare}. There is again a $U$-bound of the form \eqref{filiformubound} vanishing on the zero locus of a single coordinate $x_i$ and a vector field $X_i = \partial_{x_i}$ in $\nabla$ satisfying the estimates $\abs{X_iN} \lesssim X_iN \cdot X_i\abs{x_i}$ and $\abs{x_i}^\alpha N^{-\alpha} \lesssim X_iN \cdot X_i\abs{x_i}$ for some $\alpha \in \bbZ_{\geq 1}$. 
\end{remark}

In dimension one these $L^1$-Caffarelli-Kohn-Nirenberg inequalities allow us to achieve the endpoint case in Proposition \ref{thm3} which unfortunately the $L^1$-Hardy inequality falls just short of, but we intentionally postponed their treatment to this penultimate section since the proofs here are of a different flavour, requiring some additional conditions on the vector fields in $\nabla$ and estimates on $N$ and standing in contrast to the arguments based on partial integration and H\"older's inequality which appeared in earlier proofs.

\subsection{The anisotropic Heisenberg group}

As one final application of our ideas (however not of Theorem \ref{thm1}), we prove one more isoperimetric inequality, this time on the anisotropic Heisenberg group $\bbH^n(\frac{1}{2}, 1)$, $n \geq 2$, on which functional inequalities for measures of the form \eqref{measure} were studied in \cite{dagher2022coercive}. The group structure generalises the usual Heisenberg group $\bbH^n = (\bbR^n_x \times \bbR^n_y) \times \bbR^1_t$ structure equipped with the Heisenberg-Greiner fields \eqref{greinerfields} at $\zeta = 1$, namely $X^1_1 = \partial_{x_1} + 2\zeta y_1\partial_t$ and $Y^1_1 = \partial_{y_1} - 2\zeta x_i\partial_t$ are modified by replacing the constant $2$ with another nonzero constant. Actually, the definition given in \cite{dagher2022coercive} is different to ours; it is however isomorphic, so in the sequel we adopt their definition since we will use some of their estimates.  

Let $\bbG = \bbH^n(\frac{1}{2}, 1) \cong \bbR^{2n}_x \times \bbR^1_t$ be equipped with the vector fields $\nabla = (X_1, \cdots, X_{2n})$ given by \cite[pp.~3]{dagher2022coercive} and measures $d\mu = Z^{-1}e^{-V}d\xi$, $V = N^p$, of the form \eqref{measure} where $N^{-2n}$ is the fundamental solution of the sublaplacian $\Delta = \sum_{i=1}^{2n} X_i^2$ given by \cite[Theorem~1]{dagher2022coercive} satisfying the estimates of \cite[Lemma~2]{dagher2022coercive}. In particular, $N$ is \emph{not} a Kaplan-type homogeneous norm of the form \eqref{Nkappa} and so falls outside the scope of \hyperref[S3.1]{\S3.1}. Moreover, $N$ does not verify \eqref{estimates}. If we work backwards for instance, then we may expect because $\abs{\nabla N} \simeq \abs{x}N^{-1}$ an $L^1$-Hardy inequality 
\begin{equation}\label{anisotropichardy1}
    \int \frac{\abs{f}}{\abs{x}}d\mu \lesssim \int \abs{\nabla f}d\mu + \int \abs{f}\abs{x}N^{p-2}d\mu
\end{equation}
which leads us to seek a $U$-bound for $U_1 = \abs{x}N^{p-2}$. This follows from developing $\int \nabla(\abs{f}e^{-V}) \cdot \abs{x}^{-1}N\nabla N d\xi$ since $\Delta N \simeq \abs{\nabla N}^2N^{-1}$, giving  
\begin{equation}\label{anisotropicubound}
    \begin{aligned}
        \int \abs{f}U_1d\mu &\lesssim \int \abs{\nabla f}d\mu + \int \abs{f}U_1/N^p d\mu + \int \frac{\abs{f}}{\abs{x}}d\mu.
    \end{aligned}
\end{equation}
However it is unclear to us how to control the last addend. It is true that if one tracks dimension dependent constants, namely the fact \eqref{anisotropichardy1} carries the constant $2n-1$, there is a positive answer for $n$ sufficiently large because $\abs{\nabla N} \simeq \abs{x}N^{-1}$ with dimension free constants. To find a more satisfying answer, it turns out estimates of the form \eqref{estimates} are too rough and one needs to start with the $L^1$-Hardy inequality
\begin{equation}\label{anisotropichardy2}
    \int \frac{\abs{f}}{\abs{x}}d\mu = \int \abs{f}\hat{\nabla} \cdot \hat{\nabla}\abs{x}d\mu \lesssim \int \smallabs{\hat{\nabla} f}d\mu + \int \abs{f}N^{p-1}(\hat{\nabla}N \cdot \hat{\nabla}\abs{x})_+d\mu 
\end{equation}
where $\smash{\hat{\nabla} = (X_2, \cdots, X_n, X_{n+2}, \cdots, X_{2n})}$. Since $\smash{\hat{\nabla} N \cdot \hat{\nabla} \abs{x} = \sum_{i \neq 1, n+1} \partial_{x_i}N \cdot \partial_{x_i}\abs{x}}$, by examining \cite[Equation~2.20]{dagher2022coercive} more carefully we find
\[ \partial_{x_i}N \cdot \partial_{x_i}\abs{x} = \frac{x_i}{\abs{x}}\frac{x_i}{N}\left(\varphi_1 - \varphi_2\right), \quad i \neq 1, n+1, \]
for a pair of smooth and bounded nonnegative functions $\varphi_1$ and $\varphi_2$ homogeneous of order zero. The main point is that it is actually $\varphi_2$ which prevents us from improving our current bound $\nabla N \cdot \nabla \abs{x} \lesssim \abs{x}N^{-1}$, and that since we may verify $\varphi_1 \lesssim \abs{x}^2N^{-2} \lesssim \abs{x}N^{-1}$, we deduce \eqref{anisotropichardy2} improves \eqref{anisotropichardy1} to  
\begin{equation}\label{anisotropichardy3}
    \int \frac{\abs{f}}{\abs{x}}d\mu \lesssim \int \smallabs{\hat{\nabla} f}d\mu + \int \abs{f}\abs{x}^2N^{p-3}d\mu \lesssim \int \abs{\nabla f}d\mu + \int \abs{f}\abs{x}^2N^{p-3}d\mu.
\end{equation}
Lastly, we have the $U$-bound 
\[ \int \abs{f}\abs{x}^2N^{p-3}d\mu \lesssim \int \abs{\nabla f}d\mu + \int \abs{f}\abs{x}^2N^{p-3}/N^pd\mu, \]
which follows by developing $\int \nabla(\abs{f}e^{-V}) \cdot \nabla N d\xi$, so together with \eqref{anisotropicubound} we have all the necessary ingredients as the ones appearing in the step-two group case (at $\alpha = 1$) and therefore $\mu$ satisfies the same isoperimetric inequality. 

\begin{theorem}
    Let $\bbG \cong \bbR^{2n}_x \times \bbR^1_t$ be the anisotropic Heisenberg group $\bbH^n(\frac{1}{2}, 1)$ equipped with its Carnot-Carath\'eodory distance $d$ induced by $\nabla_\bbG$ and $N^{-2n}$ the fundamental solution of the sublaplacian as defined in \textup{\cite{dagher2022coercive}}. Then the measure $d\mu = Z^{-1}e^{-N^p}d\xi$, $p \geq 2$, satisfies the isoperimetric inequality
     \[ \cI_\mu \gtrsim \cU_{2p/(p+2)}. \]
\end{theorem}

\begin{remark}
    This is also optimal since our optimality argument for step-two groups relied just on the existence of a point $(0, t) \in (\bbR^{2n}_x \times \bbR^1_t)^*$ at which $\nabla N$ vanishes. This improves the $2$-logarithmic${}^\beta$ Sobolev inequality, $\beta = (p-3)/p$, proved in \cite[Corollary~10]{dagher2022coercive}, since the proof of our isoperimetric inequality passes through a $2$-super-Poincar\'e inequality with optimal growth $\beta_2(\epsilon) \lesssim \exp(C\epsilon^{-p/(p-2)})$ which implies $\beta = (p-2)/p$ is allowed (and optimal). Moreover, we relax their constraint from $p \geq 4$ to $p \geq 2$ and remove also their dimensional constraint $n > 5$. 
\end{remark}

Before we conclude, we note the isoperimetric inequalities presented in this paper are stable under reasonable perturbations. As seen already in the proof of the $L^1$-Caffarelli-Kohn-Nirenberg inequality, the inequality for $\mu$ was proven by using the inequality for Lebesgue measure, replacing $f$ by $fe^{-V}$, and applying a $U$-bound. It follows that replacing once again $f$ by $fe^{-W}$ in our inequalities for $\mu$ should give an isoperimetric inequality (of the same form modulo constants) for the perturbed measure $d\nu \propto e^{-W}d\mu$, provided $W$ satisfies appropriate assumptions. This is indeed true; for more details see \cite[Corollaries~2.2,~2.5,~5.3]{inglis2011u}. Extension to the infinite-dimensional setting is also possible, see, for instance, \cite[\S6]{inglis2011u}. 

\begin{acknowledgements}
\textup{We would like to thank Andreas Malliaris for helpful discussions. We are also grateful to the Hausdorff Research Institute for Mathematics for their support and hospitality where part of the work on this paper was undertaken and some of these discussions took place. The author is supported by the President’s Ph.D. Scholarship of Imperial College London.}
\end{acknowledgements}

\printbibliography

\end{document}